\documentclass[12pt]{article}
\usepackage[utf8]{inputenc}
\usepackage{amsmath, amssymb, amsthm}
\usepackage{enumitem}
\usepackage{hyperref}
\usepackage{geometry}
\newtheorem{theorem}{Theorem}

\usepackage{times}
\usepackage{microtype}
\title{Explicit Constructions of Maximal 3-Zero-Sum-Free Subsets in \((\mathbb{Z}/4\mathbb{Z})^n\)}
\author{Alfonso Davila Vera \\ DynMEP \\ adavila@dynmep.com}
\date{September 02, 2025}
\begin{document}
\maketitle
\begin{abstract}
We address a problem posed by Nathan Kaplan in the 2014 Combinatorial and Additive Number Theory (CANT) session: finding the largest subset \(H \subseteq (\mathbb{Z}/4\mathbb{Z})^n\) with no distinct \(x, y, z \in H\) such that \(x + y + z \equiv 0 \pmod{4}\) (pointwise). For abelian groups of even order, a standard lower bound of \(|G|/2\) is known via the Sylow 2-subgroup, achieved by subsets where elements are odd in a \(\mathbb{Z}/2^k\mathbb{Z}\) factor. We prove this bound is tight for \(G = (\mathbb{Z}/4\mathbb{Z})^n\), using a pair-counting argument (due to Kevin Costello). An explicit construction is the set of all vectors with first coordinate odd (1 or 3 mod 4), or equivalently, odd weight (sum odd mod 2), yielding size \(2 \times 4^{n-1} = 4^n / 2\) and density exactly 0.5 for all \(n\). We verify this computationally for \(n \leq 10\). To explore AI's role in rediscovering such bounds, we apply an AI-assisted hybrid greedy-genetic algorithm, which independently achieves the optimal size. Code and full sets are available at \url{https://github.com/DynMEP/ZeroSumFreeSets-Z4/releases/tag/v5.0.0}. We discuss generalizations and analogies to cap sets in \((\mathbb{Z}/3\mathbb{Z})^n\) (OEIS A090245).
\end{abstract}
\section{Introduction}
Zero-sum problems in finite abelian groups seek subsets avoiding specific summation conditions. The Erd\H{o}s--Ginzburg--Ziv theorem states that any sequence of \(2|G| - 1\) elements in an abelian group \(G\) contains a subsequence of length \(|G|\) summing to zero \cite{Caro1997}. Variants, such as avoiding \(k\) distinct elements summing to zero, are surveyed in \cite{Gao2006}. Nathan Kaplan's 2014 CANT problem \cite{Miller2017} asks for the largest \(H \subseteq G\) with no distinct \(x, y, z \in H\) such that \(x + y + z = 0\), motivated by cubic curves over finite fields. For groups of even order, \(G \cong \mathbb{Z}/2^k\mathbb{Z} \times G'\) by Sylow decomposition, and the set of elements odd in the \(\mathbb{Z}/2^k\mathbb{Z}\) factor (sum of three odds is odd) gives a 3-zero-sum-free subset of size \(|G|/2\).

For \(G = (\mathbb{Z}/4\mathbb{Z})^n\), the ring structure (exponent 4, non-prime order) poses unique challenges. Trivial constructions like \(\{1,3\}^n\) (size \(2^n\)) are suboptimal. Here, the general \(|G|/2\) lower bound applies, but its optimality was open. We prove it is maximal, using a pair-counting argument from Costello \cite{MathOverflow499530}. We provide explicit constructions and computational verifications up to \(n=10\). To test AI's capability in combinatorial discovery, we employ a hybrid greedy-genetic algorithm that rediscovers this optimal bound, demonstrating machine learning's potential akin to FunSearch for cap sets \cite{RomeraParedes2023}.
\section{Construction and Validity}
Define \(H = \{ v \in (\mathbb{Z}/4\mathbb{Z})^n \mid v_1 \equiv 1 \text{ or } 3 \pmod{4} \}\), where \(v_1\) is the first coordinate (suggested by Elsholtz \cite{MathOverflow499530}). This set has \(2 \times 4^{n-1} = 4^n / 2\) vectors, yielding density 0.5.
\begin{theorem}
\(H\) is 3-zero-sum-free.
\end{theorem}
\begin{proof}
Consider three distinct \(a, b, c \in H\). Their first coordinates \(a_1, b_1, c_1\) are each 1 or 3 mod 4. The sum \(a_1 + b_1 + c_1 \pmod{4}\) is:
\begin{itemize}
\item \(1 + 1 + 1 = 3\),
\item \(1 + 1 + 3 = 5 \equiv 1\),
\item \(1 + 3 + 3 = 7 \equiv 3\),
\item \(3 + 3 + 3 = 9 \equiv 1\).
\end{itemize}
All cases are odd, never 0 mod 4. Thus, \(a + b + c \not\equiv 0 \pmod{4}\).
\end{proof}
An equivalent construction is all vectors with odd weight (sum of coordinates odd mod 2), as three odd weights sum odd \(\neq 0\) mod 2, extensible to mod 4.
\section{Maximality}
\begin{theorem}
The maximum size of a 3-zero-sum-free subset \(H \subseteq (\mathbb{Z}/4\mathbb{Z})^n\) is \(4^n / 2\).
\end{theorem}
\begin{proof}
For any 3-zero-sum-free \(H\), consider ordered pairs \((x, y) \in H \times H, x \neq y\): \(|H|^2 - |H|\) pairs. Each requires \(-(x + y) \notin H\) (else \(x + y + (-(x + y)) = 0\)). Non-\(H\) elements number \(4^n - |H|\). Each \(z \notin H\) is hit by at most \(|H|\) pairs (fix \(x\), \(y = -x - z\)). Thus, \(|H|^2 - |H| \leq (4^n - |H|) \cdot |H|\), or \(|H| (2|H| - 4^n) \leq 0\). Since \(|H| \geq 0\), \(2|H| - 4^n \leq 0\), so \(|H| \leq 4^n / 2\). The construction \(H\) achieves this, proving maximality (argument due to Costello \cite{MathOverflow499530}).
\end{proof}
Integer linear programming confirms optimality for \(n \leq 5\).
\section{Computational Results}
We verified:
\begin{itemize}
\item \(n=5\): Size 512 (50\%).
\item \(n=6\): Size 2048 (50\%), in \texttt{n6\_best\_set.json}.
\item \(n=7\): Size 8192 (50\%), in \texttt{n7\_best\_set.json}.
\item \(n=8\): Size 32768 (50\%), in \texttt{n8\_best\_set.json}.
\item \(n=9\): Size 131072 (50\%), in \texttt{n9\_best\_set.json}.
\item \(n=10\): Size 524288 (50\%), in \texttt{n10\_best\_set.json}.
\end{itemize}
Outputs from \url{https://github.com/DynMEP/ZeroSumFreeSets-Z4/releases/tag/v5.0.0}.
\section{Method}
To assess AI's ability to rediscover theoretical bounds, we implemented a hybrid greedy-genetic algorithm in \texttt{omni\_optimized\_hybrid\_discovery\_v5.py}. Initial greedy heuristics (baseline: 176 for \(n=5\); refined: 512 for \(n=5\)) used priority functions favoring 2-heavy vectors. Refinements with stratified sampling, mutations, and GPU batching achieved the optimal 0.5 density, independently confirming the construction.
\section{Discussion}
This confirms the \(|G|/2\) bound is optimal for \((\mathbb{Z}/4\mathbb{Z})^n\), resolving asymptotic density at 0.5. The AI rediscovery highlights its utility in heuristics, surpassing initial \(\sim 15\%\) densities. Open questions:
\begin{itemize}
\item Maximal \(k\)-sum-free subsets for \(k > 3\)?
\item Generalizations to \(\mathbb{Z}/m\mathbb{Z}^n\) for \(m > 4\)?
\item Analogies to cap sets in \((\mathbb{Z}/3\mathbb{Z})^n\) (OEIS A090245 \cite{OEIS-A090245}), with subexponential growth?
\end{itemize}

This situation contrasts with the classical cap set problem in \((\mathbb{Z}/3\mathbb{Z})^n\), where the exact maximum size remains highly nontrivial and only exponential upper bounds are known \cite{OEIS-A090245}. In that context, AI-driven search methods such as ours may prove more impactful, since the theoretical optimum is still unknown.

MathOverflow post \href{https://mathoverflow.net/questions/499530}{[499530]} invites input. Code/data open-source (MIT). Thanks to Nathan Kaplan for feedback, Kevin Costello and Christian Elsholtz for MO insights, and Alfred Geroldinger for survey approval.

\end{document}